%% file: main.tex
\documentclass{llncs}

\pdfoutput=1
\usepackage{amsmath}
\usepackage{amssymb}
\usepackage{graphicx}
\usepackage{enumerate}
\usepackage{mathtools}
\usepackage{color}

\usepackage{cite}
\usepackage{centernot}
\usepackage{geometry}
\usepackage{array}
\usepackage{bbm}

\usepackage{tikz}
\usetikzlibrary{automata}
\usetikzlibrary{arrows,shapes,calc}
\usetikzlibrary{positioning}

\newcommand{\dsl}{\displaystyle}

\newcommand{\IsDef}{\triangleq}			

\newcommand{\Root}{\mathit{r}} 
\newcommand{\W}{\mathrm{W}} 

\newcommand{\s}{\mathit{s}} 

\newcommand{\Name}{NRRW}

\begin{document}

\title{
%
Transient and Slim versus Recurrent and Fat: \\ Random Walks and the
Trees they Grow
}

\author{ Giulio Iacobelli\inst{1} \and  Daniel R. Figueiredo\inst{2} \and Giovanni Neglia\inst{3} }

\institute{%
Instituto de Matem{\' a}tica\\Federal University of Rio de Janeiro (UFRJ), Brazil\\
\email{giulio@im.ufrj.br}	
\and
 Department of Computer and System Engineering, COPPE \\ Federal University of Rio de Janeiro (UFRJ), Brazil\\ 
 \email{daniel@land.ufrj.br}
 \and 
 NEO Team, INRIA, Sophia-Antipolis, France\\  \email{giovanni.neglia@inria.fr}
}

\maketitle

\input{abstract}


\input{introduction}

%
\input{model}

\input{s=1}

\input{Even_steps}

\input{degree_s_even}

\input{root_visits_s_even}
\input{density_s_even}

\input{conclusion}

\bibliography{RW}
\bibliographystyle{plain}


\end{document}

%% file: abstract.tex
\begin{abstract}
Network growth models that embody principles such as preferential attachment and local attachment rules have received much attention over the last decade. Among various approaches, random walks have been leveraged to capture such principles. In this paper we consider the No Restart Random Walk (\Name{}) model where a walker builds its graph (tree) while moving around. In particular, the walker takes $\s$ steps (a parameter) on the current graph. A new node with degree one is added to the graph and connected to the node currently occupied by the walker. The walker then resumes, taking another $s$ steps, and the process repeats. We analyze this process from the perspective of the walker and the network, showing a fundamental dichotomy between transience and recurrence for the walker as well as power law and exponential degree distribution for the network. More precisely, we prove the following results:
\begin{itemize}
\item[]
\begin{enumerate}
\item[$\s=1$:]  the random walk is {\em transient} and the degree of every node is bounded from {\em above} by a geometric distribution.
\item[$\s$ even:] the random walk is {\em recurrent} and the degree of non-leaf nodes is bounded from {\em below} by a power law distribution with exponent decreasing in $\s$. We also provide a lower bound for the fraction of leaves in the graph, and for $\s=2$ our bound implies that the fraction of leaves goes to one as the graph size goes to infinity.
\end{enumerate}
\end{itemize}
\Name{} exhibits an interesting mutual dependency between graph building and random walking that is fundamentally influenced by the parity of $\s$. Understanding this kind of coupled dynamics is an important step towards modeling more realistic network growth processes. 
\end{abstract}

%% file: introduction.tex

\section{Introduction}
\label{sec:intro}

Network growth models are fundamental to represent and understand the evolution of real networks, such as the web or social networks. Not surprisingly, a vast number of models that embody principles such as preferential attachment and local attachment rules have been proposed and applied over the last decade~\cite{van2016random,newman2010networks}. Capturing such principles is important not only because they tend to be present in real networks, but also because they can lead to properties such as heavy-tailed degree distribution and short distances, often observed in real networks.

A promising approach to capture such principles in a network growth model is to leverage random walks~\cite{ikeda2014network,vazquez2003growing,saramaki2004scale,cannings2013random,amorim2016growing}. In such models, the network grows as the walker moves with both processes being mutually dependent. In the Random Walk Model~\cite{saramaki2004scale,cannings2013random}, the walker takes $\s$ steps, a new node is added and connected to the current walker position, and the random walk restarts (choosing a new node uniformly at random). The ``No Restart Random Walk'' (\Name{}) Model~\cite{amorim2016growing} builds on this model by not requiring the walker to restart. In particular, \Name{} can be algorithmically described as follows:
\begin{enumerate}
\item[{\em 0.} ]
Start with a single node with a self-loop with a random walk on it.
\item[{\em 1.} ] 
Let the random walk take exactly $\s$ steps on the current graph.
\item[{\em 2.} ] 
Add and connect a new node with degree one to the current walker location.
\item[{\em 3.} ]
Go to step {\em 1}.
\end{enumerate}
Note that \Name{} is parsimonious and has a single parameter $s$, called the \emph{step parameter}. What graphs will \Name{} generate? Figure~\ref{fig:network-s} depicts graphs generated by simulating \Name{} for $\s=1$ and $\s=2$ for different number of nodes. Clearly, graphs generated for $s=1$ and $s=2$ are very different. As it turns out, the parity of $s$ and its magnitude play a fundamental role on the walker dynamics and on the graph structure. 

Interestingly, the \Name{} can be analyzed from two different perspectives: the walker, and the network. From the walker's perspective, a fundamental question is on the dichotomy between transience and recurrence~\cite{angel2014,disertori2015transience,dembo2014walking}. Will \Name{} give rise to transient or recurrent random walks? From the network perspective, a fundamental question is the dichotomy between a heavy-tailed and exponential degree distribution. What kind of degree distribution will \Name{} generate? Given the mutual dependency between walker and network the two perspectives are fundamentally intertwined. 

Indeed, Theorem~\ref{thm:s_1} illustrates this close relationship between walker and network characteristics, establishing that for $\s=1$ the random walk is {\em transient} and the degree distribution is upper bounded by an exponential distribution. Most interestingly, these observations on the walker and the network emerge together from the model dynamics without one  being the cause for the other! For $\s$ even we show that the walker is {\em recurrent} and the degree distribution of non-leaf nodes is bounded from below by a power-law distribution with exponent decreasing in $\s$. Moreover, for $\s$ even we also provide a lower bound for the fraction of leaves in the graph, and for $\s=2$ our bound implies that the fraction of leaves goes to one as the graph size goes to infinity. 

Why is the parity of $\s$ so fundamental? To help build intuition, consider $s=1$. In this case, the walker can always move to the node just added, and consequently connect a new node to it. For $s=2$, the walker can never connect a new node to the just added node (except if the self-loop is traversed). These two observations, the runaway effect and bouncing back effect, respectively, play a major role on the dynamics of \Name{}, as soon discussed.

\begin{figure}[t]
\begin{center}
\scalebox{0.60}{
\begin{tabular}{|c|m{6cm}|m{6cm}|m{6cm}|}
 \hline
& & & \\
\begin{tabular}{c}
{\large \rm step} \\ {\large \rm parameter}
\end{tabular}
{\large   \textbackslash }
\begin{tabular}{c}
{\large \rm graph } \\ {\large \rm size}
\end{tabular} 
 &
\begin{center}
 {\large $N=10^2$ }
\end{center}  
  &  
\begin{center}
 {\large $N=10^3$}
\end{center}   
& 
\begin{center}
 {\large $N=10^4$}
\end{center}  
\\
 \hline
{\Large $\s=1$} 
 &
\includegraphics[scale=0.27]{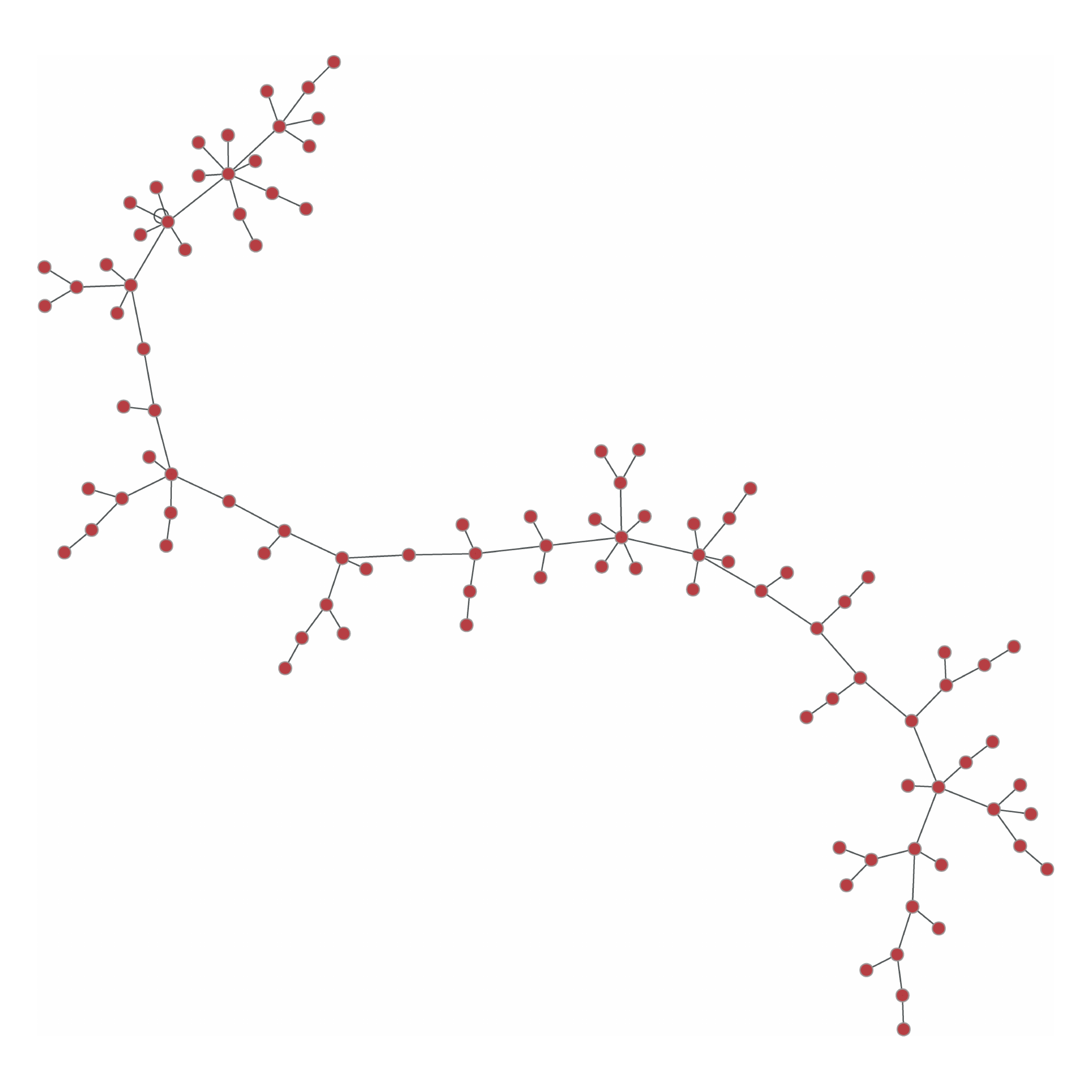}
  & 
 \includegraphics[scale=0.27]{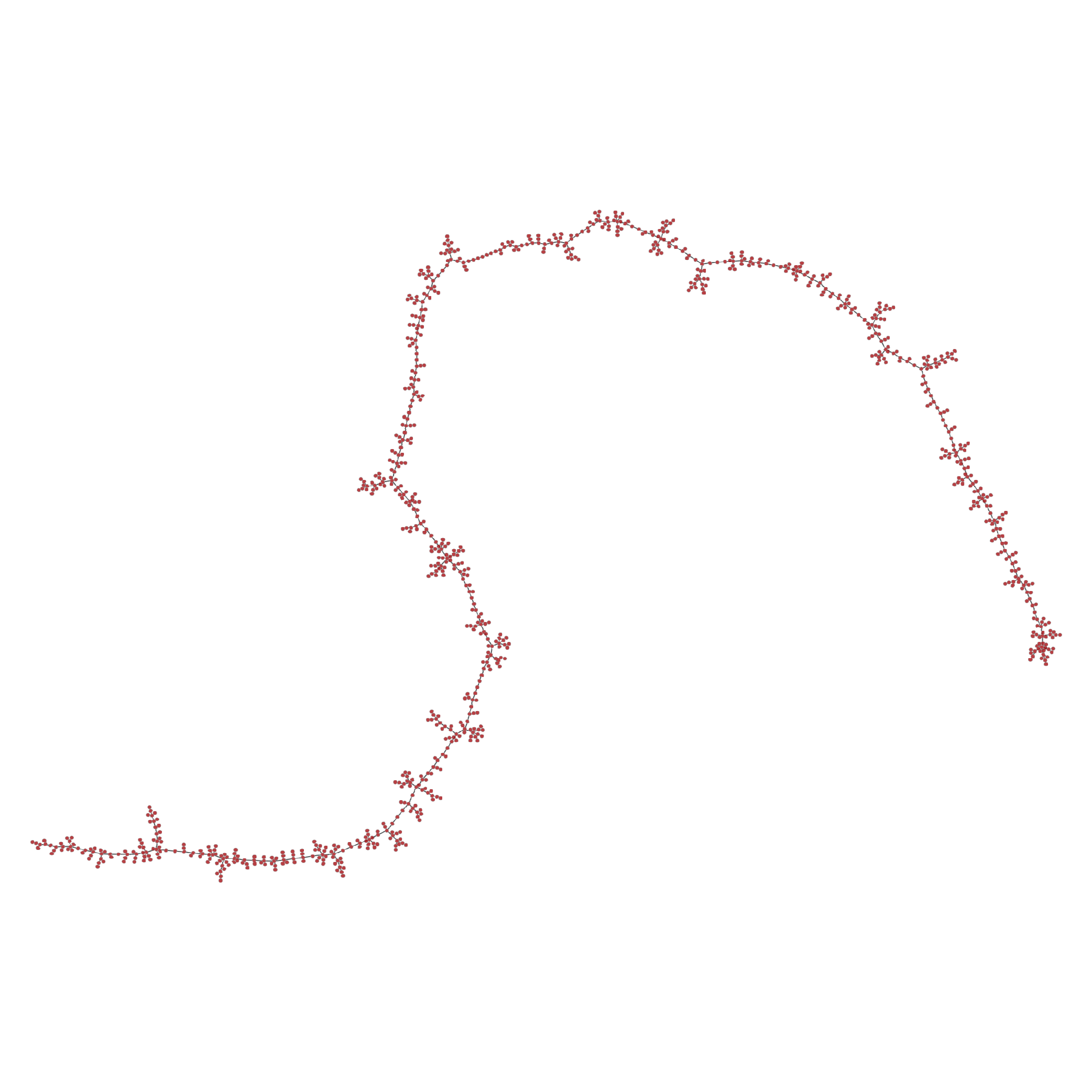} 
  & 
 \includegraphics[scale=0.29]{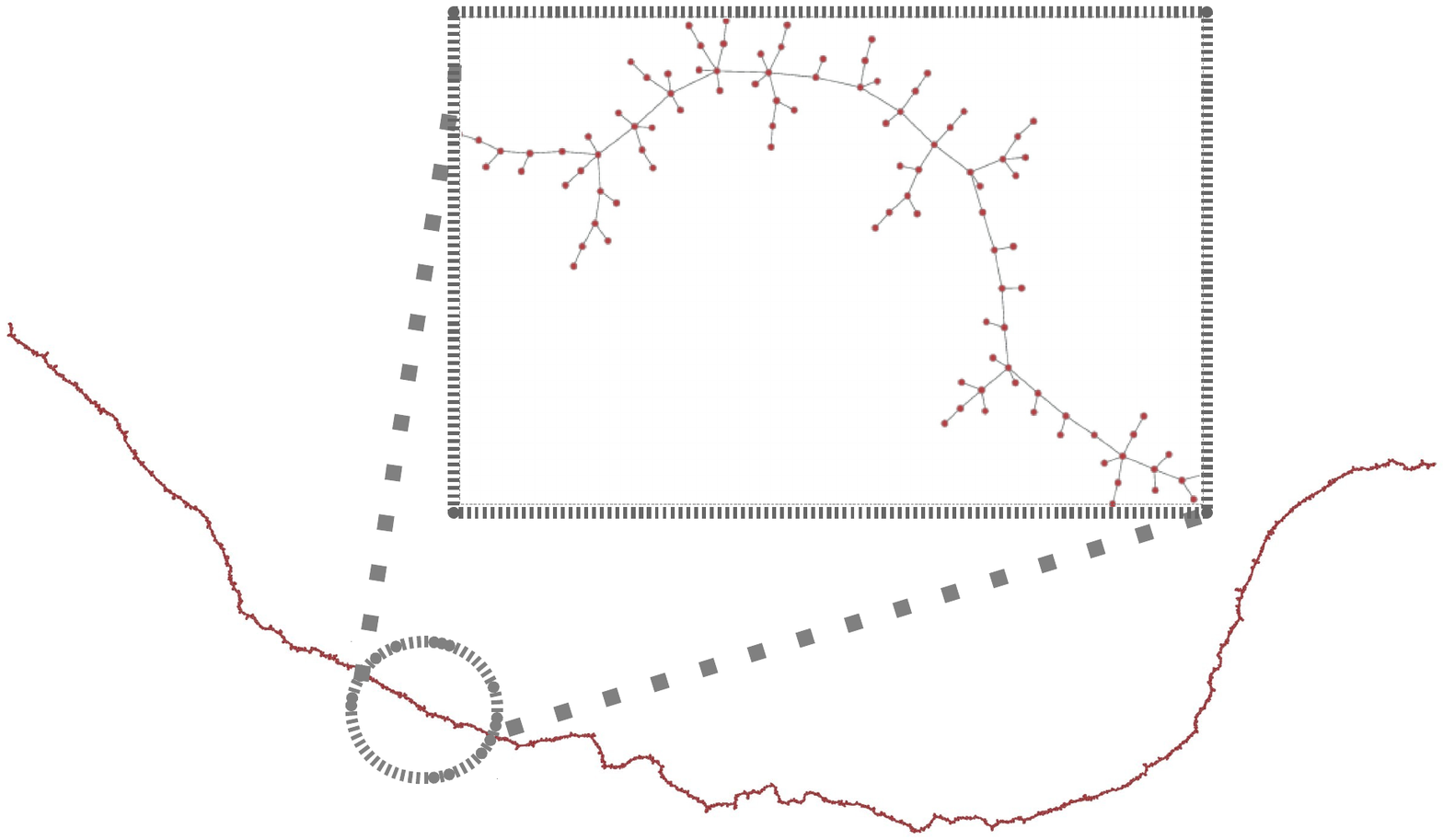}   
  \\
\hline
 {\Large $\s=2$} & 
  \includegraphics[scale=0.27]{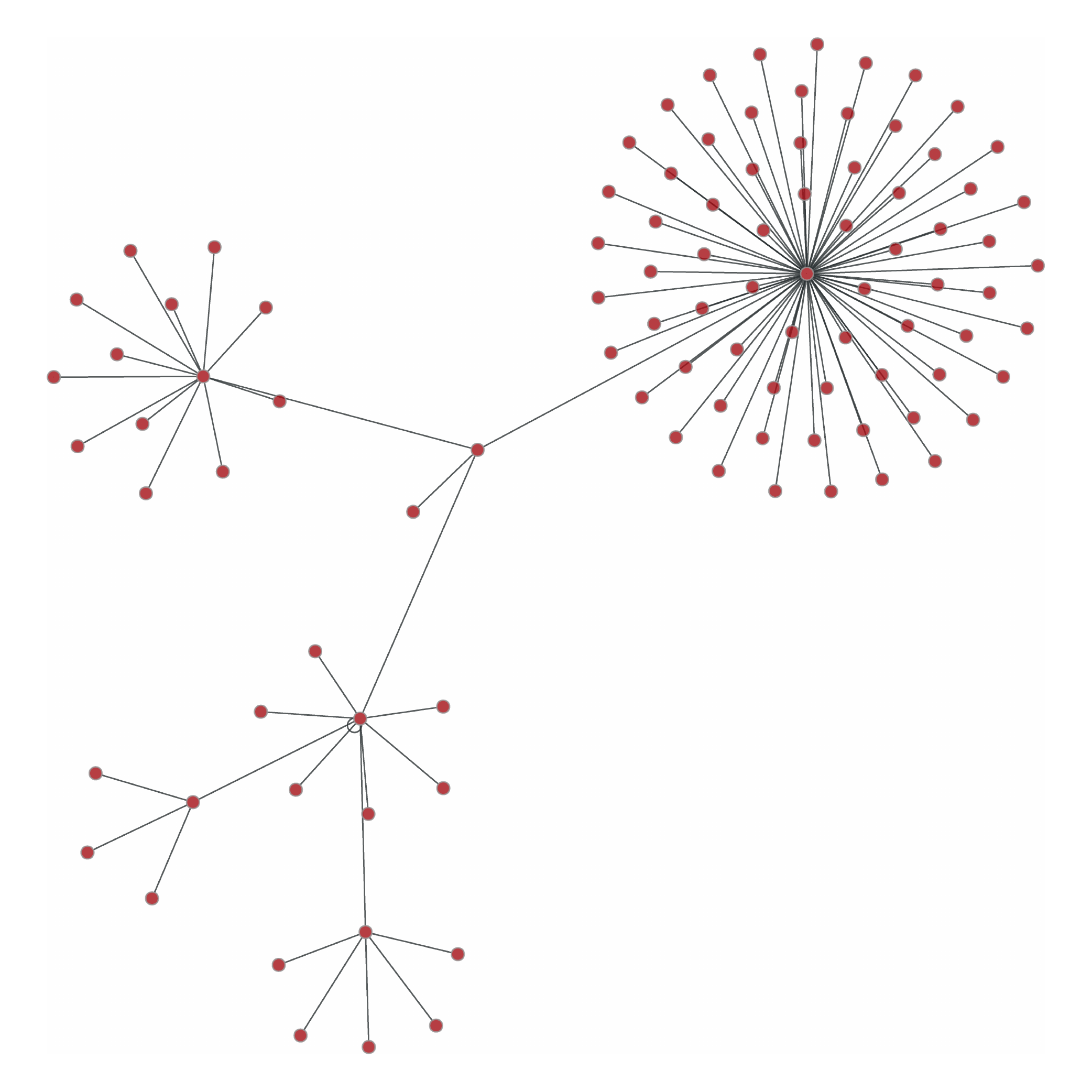}   
 &
  \includegraphics[scale=0.27]{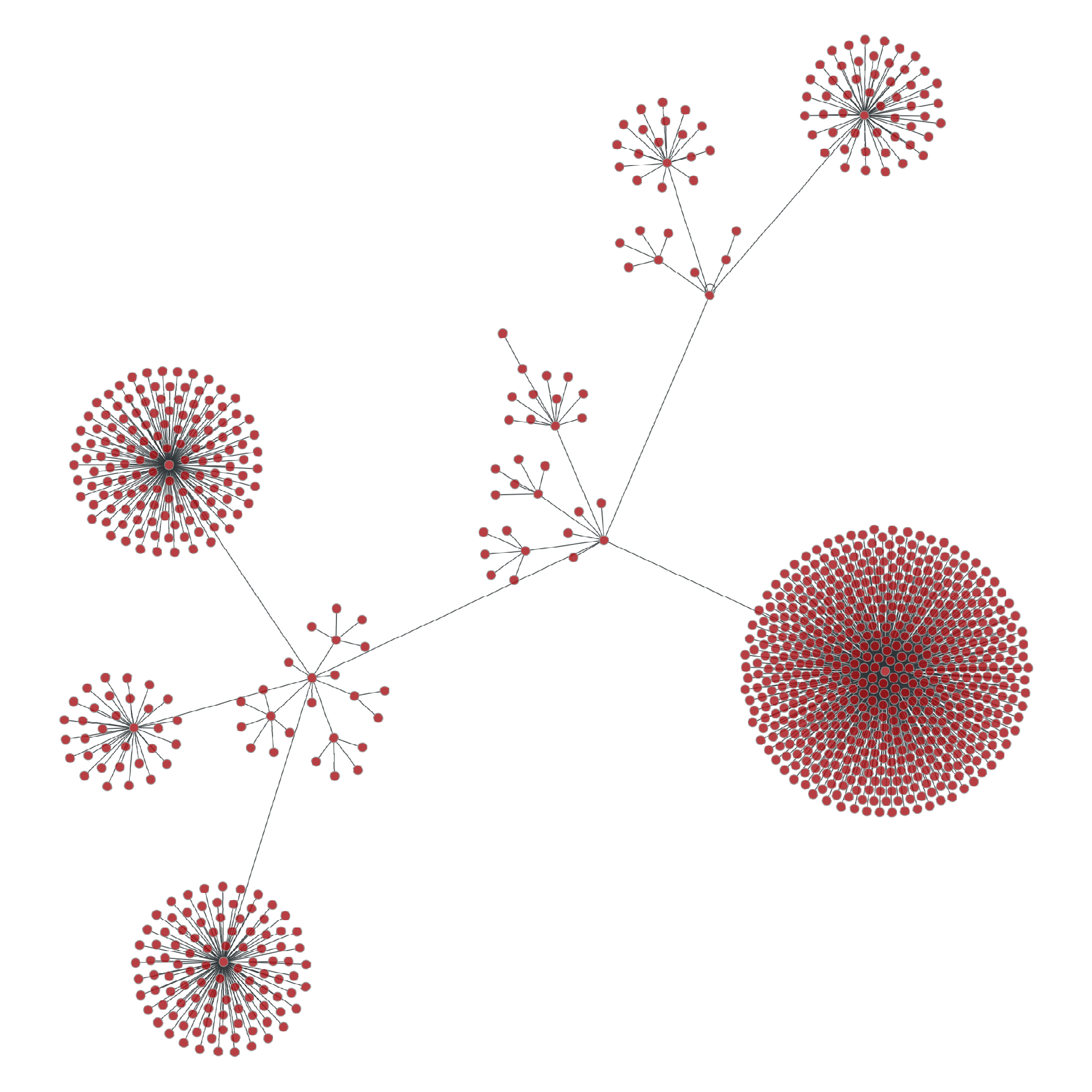}
  &   
 \includegraphics[scale=0.27]{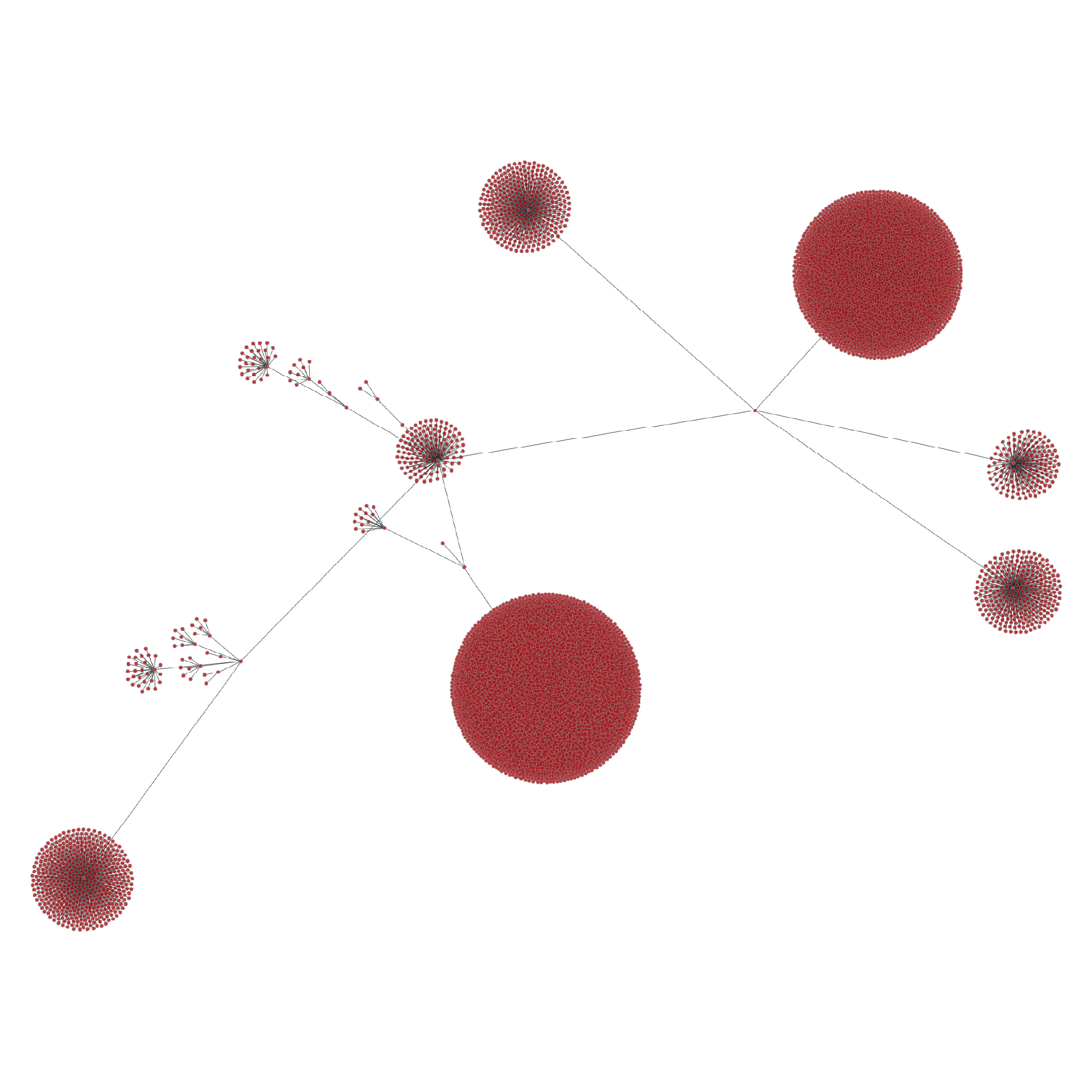}
 \\
\hline
\end{tabular}
}
\caption{Graph  generated by simulating \Name{}  with step parameter $\s=1$ and $\s=2$ for different times,  $t= \s (N-1)$, with  $N= 10^2,10^3, 10^4$.}
\label{fig:network-s}
\end{center}
\end{figure}

Last, various properties of networks generated by \Name{} were first observed by means of extensive numerical simulations~\cite{amorim2016growing}, such as the dichotomy in the degree distribution and distance distribution as a function of the parity of $\s$. In this paper we provide rigorous theoretical treatment for some of the results observed empirically. 


The remainder of the paper is organized as follows. Section \ref{sec:model} presents the model and some notation. Results for $\s=1$ case are presented in Section \ref{sec:s_is_one}, while results for $\s$ even are presented in Section \ref{sec:even_case}. Last, Section~\ref{sec:conclusion} concludes the paper with a brief discussion. 

%
%





%

%% file: model.tex

\section{The Model}\label{sec:model}

The No Restart Random Walk (\Name{}) model consists of a random walk moving on a graph that itself grows over time. At time zero, the network has a single vertex with a self-loop, called the root, with the random walk on it. At every discrete time $t>0$, the random walk takes a step in the current graph. After exactly $\s$ steps, a new vertex with degree one joins the graph, and is attached to the vertex currently occupied by the walker (we assume that the attachment of a new vertex takes zero time). 

Our stochastic process is specified by the pair $\{G^t_\s, W_t\}_{t\geq 0}$. 
 $G^t_\s$ denotes the graph process resulting at time $t$ from the \Name{} with step parameter $\s$: is an undirected graph on the vertex set $V_t = \{j\;, 0\leq j \leq \lfloor t/s \rfloor \}$, where $j$ is the label of the vertex added at time $j \s$. Note that the graph only changes at times $t= k \s$, for every integer $k>0$, while it does not change in the time intervals  $k\s \leq t < (k+1)\s$.  $W_t \in V_t$ denotes the position of the walker at time $t\in \mathbb{N}$. Note that we consider a symmetric random walk, which chooses its next step uniformly at random among the neighbors of its current position. 
Also, since a new vertex is connected to a node in the graph through a single edge, the model always grows trees. Node $0$ is the root of the tree, and we shall also denote it  by $\Root$.

The stochastic process $\{G^t_\s, W_t\}_{t\geq 0}$ is always transient: a given graph can be observed for at most $s$ consecutive steps. Nevertheless, we can define transience and recurrence of the randow walk $W_t$ similarly to what is done for a random walk on a static graph. The walker is recurrent (resp. transient) if it visits any node an infinite number of times with probability $1$ (resp. 0).

It is worthwhile noting that, while we consider as initial graph a single node with a self-loop, the results can be generalized to any initial graph, when the step parameter is odd, and to any not bipartite graph when it is even. The self-loop guarantees that the initial graph is not bipartite. The importance of the self-loop will be further discussed in Section~\ref{sec:change-of-parity}, for the moment we can observe that if we start from a single node, without the self-loop, the model with an even step parameter will trivially give rise to a star graph with the root being the star center.

%% file: s=1.tex

\section{\Name{}  with step parameter s=1}
\label{sec:s_is_one}

Recall that when $\s=1$, after every walker step a new vertex is added to the graph. 
The graphs shown in Figure~\ref{fig:network-s} (top plots) indicate that in this scenario the trees grow in depth as the number of vertices increases. A more substantial empirical evidence of this phenomenon is provided in~\cite{amorim2016growing}. 
%
This suggests that the random walk is extending the tree to lower depths just never to return to its origins. In other words, the random walk is transient and visits each vertex in the tree only a relative small number of times, with high probability. 
The following theorem captures this intuition.



\begin{theorem}
In the \Name{} model with $s=1$, the number of random walk visits to a vertex 
 is stochastically dominated by a geometric random variable and then the  random walk is transient. That is, let $J_i = \sum_{t=1}^\infty \mathbbm{1}(W_t=i)$ denote the number of visits to vertex $i$ and $t' \in \mathbb{N}$, 
  it holds 
\label{thm:s_1}
\begin{align*}
\mathbb{P} & (J_i \geq k  \mid W_{t'}=i ) \leq \lambda^k\;, \qquad \qquad \mbox{for some }0<\lambda<1\;.\\
\mathbb{P} & (W_t=i \; \text{ for infinitely many $t$ } \mid W_{t'}=i )=0\;.
\end{align*}

\end{theorem}


A theorem with a similar statement appeared in \cite{amorim2016growing}, but we restate the claim and proof here for clarity and for completeness.
The proof uses a coupling argument to show that conditioned on the random walk being in a vertex $i$ at a given time $t_0$, the number of visits to $i$ for $t > t_0$ is stochastically dominated by the number of visits a biased random walk on $\mathbb{Z}_{\geq 0}$ makes to the origin (conditioned on starting at the origin). The latter is known to be a geometric random variable (supported on $\mathbb{Z}_{\geq 0}$) with  parameter $1 - f_0$, where $f_0$ denotes  the probability that the first return time to the origin is finite. The biased random walk on $\mathbb{Z}_{\geq 0}$ is transient and, thus $f_0<1$.

\begin{proof}
We consider here that the initial graph consists of a single vertex with no self-loop. This simplifies the notation and does not compromise the main argument.
Let $r$ denote the initial vertex (the root) of the growing graph 
where the walker resides at time zero. Note that $r$ is the only vertex at level zero. 
Let $X_n$ be the level (i.e.~the distance from the root) of the vertex where the walker resides after taking $n>0$ steps. 
We call the process $\{X_n, n \in \mathbb{Z}_{\ge 0}\}$ the \emph{level process}. 
Note that the random walk visits $r$ the same number of times that the level process visits level zero. 
At step $n>0$ the RW is in a vertex $v_n$ with at least two edges: the one the RW has arrived from and the new one added as a consequence of the RW's arrival to that vertex (recall that for $\s=1$ after every walker step a new vertex is attached to the graph). Let $d_n \ge2$ denote the degree of vertex $v_n$. If $v_n \neq r$, the RW moves from $v_n$ to a vertex with a larger level with probability 
$\frac{d_n-1}{d_n}\ge \frac{1}{2}$ and with the complementary probability $\frac{1}{d_n}\le \frac{1}{2}$ to a vertex with smaller level (i.e., the parent of $v_n$ in the tree). If $v_n=r$, then  the level can obviously only increase.
Note that the nodes' degrees keep changing due to the arrival of new nodes (edges), and therefore the level process is non-homogeneous (both in time and in space). 

We now study the evolution of the level process every two walker steps, i.e.~we consider the process $Y_n \triangleq X_{2n}$. Given that the network is a tree and $X_0=0$, the two-step level process can be seen as a non-homogeneous `lazy' random walk on $2\mathbb{Z}_{\ge 0} = \{0,2,4,\ldots\}$.
We denote by $p_{k,h}(n)$ the probability that the level at step $n+1$ is $h$ conditioned on the fact that it is $k$ at step $n$. Although the notation hides it, we observe that  the probabilities $p_{k,h}(n)$ depend on the whole history of the RW until step $n$. The two-step level process will provide a bound to the transition probabilities $p_{k,h}(n)$ that allows a simple coupling with a homogeneous (and biased) random walk. 
%
%
%
The above bounds for $X_n$ lead immediately to conclude that $p_{k,k+2}(n)\ge\frac{1}{2} \frac{1}
{2}=\frac{1}{4}$ for any level $k \geq 0$ and $p_{k,k-2}(n)\le \frac{1}{2} \frac{1}{2}=\frac{1}{4}$ for $k\ge 2$, but we can get a tighter bound for  $p_{k,k-2}(n)$. If the RW is at level $k$, all the vertices on the path between its current position and the root $r$ have degree at least $2$. If it then moves to vertex $v$ at level $k-1$, a new edge is attached to $v$, whose degree is now at least $3$. The probability to move from $v$ further closer to the root to a vertex with level $k-2$, is then at most $\frac{1}{3}$. It follows then that $p_{k,k-2}(n)\le \frac{1}{2} \frac{1}{3}=\frac{1}{6}$ for $k\ge 2$. 

We consider now a homogeneous biased lazy random walk $(Y^*_n)_{n\geq 0}$ on $2\mathbb{Z}_{\ge 0}$ starting from $0$ with transition probabilities $p^*_{k,k+2}=\frac{1}{4}$ for all $k \in 2\mathbb{Z}_{\ge 0}$ and $p^*_{k,k-2}=\frac{1}{6}$ for $k \in 2\mathbb{Z}_{\ge 0}$ and $k\neq 0$. We show that if $(Y^*_n)_{n\geq 0}$ also starts in $0$ ($Y^*_0=0$), it is stochastically dominated by $(Y_n)_{n\geq 0}$. We prove it by coupling the two processes as follows.
Let $(\omega_n)_{n\geq 0}$ be a sequence of independent uniform random variables over $[0,1]$. We use them to generate sample paths for both processes $(Y_n)_{n\geq 0}$ and $(Y^*_n)_{n\geq 0}$ as follows:
\[
\begin{aligned}
    Y_{n+1}= 
\begin{cases}
    Y_n-2,& \text{if } \omega_n \in [0,p_{k,k-2}(n))\\
    Y_n+2,& \text{if } \omega_n \in [1-p_{k,k+2}(n),1]\\
    Y_n              & \text{otherwise}
\end{cases}
\end{aligned}
\;\;\;
\begin{aligned}
    Y^*_{n+1}= 
\begin{cases}
    Y^*_n-2,& \text{if } \omega_n \in [0,p^*_{k,k-2})\\
    Y^*_n+2,& \text{if } \omega_n \in [1-p^*_{k,k+2},1]\\
    Y^*_n              & \text{otherwise}
\end{cases}
\end{aligned}
\]
where $p_{k,k-2}(n)$ and $p^*_{k,k-2}$ are $0$ if $k=2$.
We start observing that if $Y_n$ and $Y^*_n$ have the same value $k$, then every time $Y^*_n$ increases also $Y_n$ increases because $p^*_{k,k+2}=\frac{1}{4}\le p_{k,k+2}(n)$. On the contrary if $Y^*_n$ decreases (as it can happen only for $k\ge2$), then $Y_n$ may decrease or not because $p_{k,k-2}(n)\le \frac{1}{6}=p^*_{k,k-2}$. It follows that if $Y_n$ and $Y^*_n$ are at the same level, then $Y^*_{n+1} \le Y_{n+1}$. 

We now prove by induction on $n$ that $Y^*_{n+1} \le Y_{n+1}$ for every $n$.
With a slight abuse of terminology we say that $Y_n$ increases (resp. decreases) if $Y_{n+1}>Y_n$ (resp. $Y_{n+1}<Y_n$).
We start observing that indeed $Y^*_0 \le Y_0$, because both  processes start in $0$.
Let us assume that $Y^*_n=h \le k= Y_n$. For all values of $h$, every time $Y^*_n$ increases also $Y_n$ increases because $p^*_{k,k+2}=\frac{1}{4}\le p_{k,k+2}(n)$ and then $Y^*_{n+1}=h+1 \le k+1=Y_{n+1}$.
If $h\ge 2$, then  $p^*_{h,h-2}=\frac{1}{6}\ge p_{k,k-2}(n)$ and if $Y_n$ decreases then $Y^*_n$ must also decrease ($Y^*_{n+1}=h-1 \le k-1=Y_{n+1}$). It follows that for $h \ge 2$ then $Y^*_{n+1} \le Y_{n+1}$.
The only case when $Y_n$ may decrease without $Y^*_n$ decreasing is when $h=0$ and $k \neq 0$, but in this case $Y^*_{n+1} =0$ and $Y_{n+1} \geq 0$. This proves that $Y^*_{n+1} \le Y_{n+1}$ for every $n$.

%
Given that $Y_n^* \le Y_n$ and both processes start at level zero, the number of visits of $(Y_n)_{n\geq 0}$ to level zero is bounded by the number of visits of $(Y^*_n)_{n\geq 0}$ to level zero. The homogeneous biased lazy random walk $(Y_n^*)_{n\geq 0}$ is transient since
$p^*_{k,k+2}=1/4 > p^*_{k,k-2}=1/6$. Thus, if $f_0$ denotes the probability that the first return time to level 0 is finite, it holds  $f_0<1$. 
By the strong Markov property, the number of visits to level 0 is geometrically distributed on the set $\mathbb{Z}_{\geq 0}$ with parameter equal to $1 - f_0$. Since a visit to level zero in $(X_n)_{n\geq 0}$ (one level process) implies a visit to level zero in $(Y_n)_{n\geq 0}$ (two level process), then it follows that the number of visits of $(X_n)_{n\geq 0}$ to level zero is bounded by a geometric random variable and then even more so by $1$ plus the same geometric random variable.

Now let us consider any vertex $v$ in the growing graph. If the RW never visits $v$, then the number of visits is $0$ (the degree of $v$ is $1$) and the thesis follows immediately. Otherwise,  consider the first time the RW visits $v$ to be time $t=0$ and  consider $v$ to be the root of the current tree. We can retrace the same reasoning and conclude that the number of visits to $v$ for $t>0$ is bounded by a geometric random variable on $\mathbb{Z}_{\geq 0}$ with parameter equal to $1 - f_0$. Then the total number of visits to $v$ is bounded by $1$ plus such random variable. This concludes the proof.
\qed
\end{proof}

\begin{corollary}
In \Name{} with $s=1$, the degree distribution of any vertex, conditioned on the random walk visiting the vertex at least once, is bounded above by a geometric distribution. 
\label{cor:s_1}
\end{corollary}
This follows because when $\s=1$ a new vertex is added after every walker step. Thus, the degree of a vertex is equal to $1$ plus the number of visits the random walk makes to the vertex. After the random walk visits a vertex for the first time, the number of subsequent  visits is stochastically dominated by a geometric random variable with support on $\mathbb{Z}_{\geq 0}$.

%

%% file: Even_steps.tex

\section{\Name{} with even step parameter}\label{sec:even_case}

In this section we focus on the behavior of \Name{} for  $s$  even, and show some fundamental difference with the case $\s=1$. In particular, we prove the following results:
\begin{itemize}


\item[-] The random walk is recurrent, i.e., visits every vertex of  the graph infinitely often almost surely. 

\item[-] The degree distribution of a vertex with degree at least two is lower bounded by a power law. 

\item[-] The fraction of leaves is asymptotically lower bounded by a constant. In particular, for $\s=2$ the fraction of leaves goes to $1$.

\end{itemize}

An important difference between the case $\s$ even and  $\s$ odd is that, when $\s$ is even, two consecutive vertices can be connected to the same vertex, something not possible with $s=1$ (or $s$ odd, in general), but for the case of the root. 
 Moreover, the more vertices are consecutively  connected to vertex $i$, the higher is the probability that the next vertex will also be connected to $i$. This produces what we call the \emph{bouncing-back effect} which increases the probability of returning to $i$ after an even number of steps. 
Assume, for example, that $\s=2$ and that $M$ new vertices have been consecutively connected to vertex $i$. All the $M$ newly added vertices are leaves and therefore as soon as the walker visits anyone of them (which happens with probability proportional to $M$), it must return to vertex $i$ in the next step, thus adding a further leaf to $i$. 
The bouncing-back effect is the fundamental difference between the dynamics of $\s$ even and odd. Last, we note that this effect is related to ``cumulative advantage" or ``rich-gets-richer" effects, since more resources an agent has (leaves of a node, in our case) the easier it becomes to accumulate further resources.

%% file: degree_s_even.tex
\subsection{Degree distribution of a vertex with $\s$ even}\label{sec:power-law}
We have seen in Corollary~\ref{cor:s_1} that for $\s=1$ the degree distribution of non-leaf vertices (i.e., vertices having degree greater than  one) is bounded from \emph{above} by a geometric distribution. In sharp contrast, we show that for every $\s$ even the degree distribution of non-leaf vertices is bounded from \emph{below} by a power-law distribution whose exponent depends on $\s$.
The dichotomy between the exponential tail and the heavy tail for the degree distribution in \Name{} for $\s$ odd and $\s$ even was already empirically observed in simulations~\cite{amorim2016growing}.

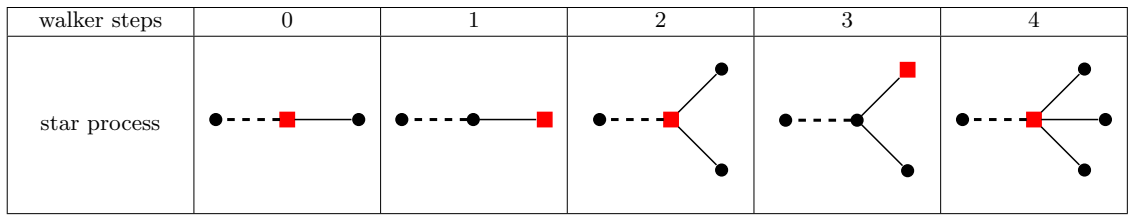
\begin{figure}
\begin{center}
\scalebox{0.95}{
\begin{tabular}{|m{2.5cm}| m{2.5cm}|m{2.5cm}|m{2.5cm}|m{2.5cm}|m{2.5cm}|}
\hline
\centering
{\rm walker steps} &  \centering $0$ &  \centering $1$ & \centering $2$ & \centering $3$ & \centering $4$ \cr 
 \hline
 \centering
{\rm star process}
&
\begin{center}
\begin{tikzpicture}[scale=1, transform shape, shorten >=0.5pt,node distance=5cm, semithick]

\tikzstyle{every state}=[scale=0.2,draw, fill]

\node[state,minimum size=10mm,shape=rectangle,
draw, fill, red] (1b)  {};
\node[state] (2b) [left of=1b] {};
\node[state] (3b) [right of=1b] {};


\draw[very thick , dashed] (1b) -- (2b);
\draw (1b) -- (3b);
\end{tikzpicture}
\end{center}
&
\begin{center}
\begin{tikzpicture}[scale=1, transform shape, shorten >=0.5pt,node distance=5cm, semithick]

\tikzstyle{every state}=[scale=0.2,draw, fill]

\node[state] (1b)  {};
\node[state] (2b) [left of=1b] {};
\node[state,minimum size=10mm,shape=rectangle,
draw, fill, red] (3b) [right of=1b] {};


\draw[very thick , dashed] (1b) -- (2b);
\draw (1b) -- (3b);
\end{tikzpicture}
\end{center}
&
\begin{center} 
\begin{tikzpicture}[scale=1, transform shape, shorten >=0.5pt,node distance=5cm, semithick]

\tikzstyle{every state}=[scale=0.2,draw, fill]

\node[state,minimum size=10mm,shape=rectangle,
draw, fill, red] (1b)  {};
\node[state] (2b) [left of=1b] {};
\node[state] (3b) [above right of=1b] {};
\node[state] (4b) [below right of=1b] {};


\draw[very thick, dashed] (1b) -- (2b);
\draw (1b) -- (3b);
\draw (1b) -- (4b);
\end{tikzpicture}
\end{center}
&
\begin{center} 
\begin{tikzpicture}[scale=1, transform shape, shorten >=0.5pt,node distance=5cm, semithick]

\tikzstyle{every state}=[scale=0.2,draw, fill]

\node[state] (1b)  {};
\node[state] (2b) [left of=1b] {};
\node[state,minimum size=10mm,shape=rectangle,
draw, fill, red] (3b) [above right of=1b] {};
\node[state] (4b) [below right of=1b] {};


\draw[very thick, dashed] (1b) -- (2b);
\draw (1b) -- (3b);
\draw (1b) -- (4b);
\end{tikzpicture}
\end{center}
&
\begin{center}
\begin{tikzpicture}[scale=1, transform shape, shorten >=0.5pt,node distance=5cm, semithick]

\tikzstyle{every state}=[scale=0.2,draw, fill]

\node[state,minimum size=10mm,shape=rectangle,
draw, fill, red] (1b)  {};
\node[state] (2b) [left of=1b] {};
\node[state] (3b) [right of=1b] {};
\node[state] (4b) [above right of=1b] {};
\node[state] (5b) [below right of=1b] {};

\draw[very thick, dashed] (1b) -- (2b);
\draw (1b) -- (3b);
\draw (1b) -- (4b);
\draw (1b) -- (5b);
\end{tikzpicture}
\end{center}
%
%
%
\\
\hline 
\end{tabular}
}
\caption{Illustration of configurations from the star growing process used to bound the degree distribution of non-leaf vertices in  \Name{}  with even step-parameter (Proposition~\ref{pro:degree-power_law}). The dashed edge represents the edge to the parent node, solid edges represent edge to leaves, and the red (squared) vertex represents the  walker position at the corresponding time.}
\label{fig:star-process}
\end{center}
\end{figure}

\begin{proposition}\label{pro:degree-power_law}
Let $\s$ be even and assume that $T_j$ is the first time a vertex is connected to vertex $j$.  
Then, for every time $t \geq T_j$ and for every $k \in \{1, \dots, \lfloor \frac{t-T_j}{\s}\rfloor +1\}$  we have that
\begin{align*}
&\mathbb{P}\left(d_{t}(j)\geq k+1  \mid T_j < \infty\right)  \geq
k^{-\s/2}\;, && \text{ if $j \neq $ root} \:, \\
& \\
&\mathbb{P}\left(d_{t}(j)\geq k+2\right)  \geq \left( \frac{k(k+1)}{2}\right)^{-\s/2}\;, && \text{ if $j = $ root}\;,
\end{align*}
where $d_t(j)$ denotes the degree of vertex $j$ at time $t$.
\end{proposition}


\begin{proof}
To prove the claim we compute the degree distribution of a node in a much simpler process. 
This simpler process starts at time $0$ with a 3-node graph made by a node $c$ that has exactly one child (leaf) and one parent and the random walk placed on node $c$. If at any step the random walk chooses the parent node, the process stops. Otherwise, after $s$ steps the random walk adds 
a new leaf node. Note that since $\s$ is even, this simple process will grow a star. The star stops growing when the random walk steps into the parent node. Figure~\ref{fig:star-process} illustrates this simple growing star process.

At time $t>0$ the random walk can have added at most $\lfloor \frac{t}{\s}\rfloor$ nodes to $c$. In particular, it adds one node if for $s$ steps consecutive, it selects a child of $c$ (this happens with probability $i/(i+1)$ if $i$ is the current degree of node $c$) and then steps back to $c$. It follows that the probability that the degree of node $c$ at time $t$ (we denote it by $d_t(c)$) is at least $k+1$ for $k\in \{1, \dots, \lfloor \frac{t}{\s}\rfloor +1\}$ is equal to: 
\begin{align*}
\mathbb{P} \left(d_t(c)\geq k+1\right)&= \prod_{i=1}^{k-1}\left(\frac{i}{i+1}\right)^{\frac{\s}{2}}= \left(\frac{1}{k}\right)^{\s/2}\;,
\end{align*}
with the usual convention that $\prod_{i=1}^{0}\left(\frac{i}{i+1}\right)^{\frac{\s}{2}}=1$. 

\medskip

Let $j$ be an arbitrary vertex of the graph generated by the \Name{} model with even step parameter $\s$ and assume that $T_j < \infty$ (recall that $T_j$ is the first time a new node is connected to node $j$).
%
Let us first consider the case in which  $j$ is different from  the root.  At time $T_j$ vertex $j$ has exactly two neighbors; a leaf and a parent. Thus, at this point in time, the dynamics of the \Name{} model is similar to that of the simple star growing process described above. 

In particular, the probability that at time $t \geq T_j$, the degree $d_{t}(j)$ is larger than $k+1$ for $k \in \{1, \dots, \lfloor \frac{t-T_j}{\s}\rfloor+1\}$ is greater than the probability that the corresponding probability for the simple star, because the new nodes can be added to $j$ without being consecutive and in particular nodes can still be added after the random walk steps on to the parent node.
%

In case $j$ is the root ($j=\Root$), we have that $T_j=\s$. Moreover, we can treat the self-loop at the root as the edge leading to the parent node. 
In order to bound the root degree distribution we consider a slight variation of the star growing process since the initial configuration has a self-loop. Consequently, the probability the walker takes the leaf node in the first step equals $1/3$ (rather 
than $1/2$) and the probability it takes the parent in the first step is $2/3$.
If $c_\Root$ denotes the center vertex of this variant of the growing star process, similarly to the previous case, for 
$k\in \{1, \dots, \lfloor \frac{t}{\s}\rfloor +1\}$, we have
\begin{align*}
\mathbb{P} \left(d_{t}(c_\Root)\geq k+2\right)&= \prod_{i=1}^{k-1} \left( \frac{i}{i+2}\right)^{\frac{\s}{2}}=\left( \frac{2}{k(k+1)} \right)^{\frac{\s}{2}}\;.
%
\end{align*}
and then the lower bound  for the degree $d_t(r)$ is derived through the same reasoning.

\qed
\end{proof}

Albeit the presence of  the bouncing-back effect,  it is worth noticing that the random walk will not get stuck  going back and forth from a vertex to its leaves, and it will eventually stop bouncing. 
In particular, we show that for any $\s$ even,  the probability of bouncing back $k$  times goes to zero faster than $k^{-1/2}$.
%
%

\begin{lemma}\label{lem:star-escape}
Let $\s$ be even,  $t_0 \in 2\mathbb{Z}_{\geq 0}$ and $t_0 \ge s$, and $i$ a vertex of the graph. Then, for all $k\geq 1$ it holds 
\[
\mathbb{P}\big(\W_{t_0 +2}=i, \W_{t_0 +4}=i,\cdots, \W_{t_0 +2k}=i \mid \W_{t_0}=i\big)\leq
\begin{cases}
 2 \frac{\sqrt{d_{t_0}\!(i)-1}}{\sqrt{d_{t_0}\!(i) + k -1}} \;, & \text{ if $d_{t_0}\!(i)\geq 2$, } 
 \\
 \frac{1}{\sqrt{k}} & \text{ if $d_{t_0}\!(i)=1$, } 
 \end{cases}
\] 
where $d_{t_0}\!(i)$ is  the degree of vertex $i$ at time $t_0$.
\end{lemma}

\begin{proof}
%
%
%
Let $\mathcal{N}_{t_0}\!(i)$ denote the set of neighbouring  vertices of  $i$ at time $t_0$. The probability  the  walker returns to $i$ at time $t_0 + 2$ given that $\W_{t_0}=i$ is equal to $B_{t_0}(i)/d_{t_0}(i)$, where
\begin{equation*} 
B_{t_0}(i)=
\begin{cases}
\dsl\sum_{j \in \mathcal{N}_{t_0}(i)}\frac{1}{d_{t_0}(j)}  \;, 
&
\text{if $i\neq \Root$\;,} 
\\
\dsl\sum_{j \in \mathcal{N}_{t_0}(\Root) \setminus \{\Root\}}\frac{1}{d_{t_0}(j)} + \frac{4}{d_{t_0}(\Root)} \;
 &
\text{if $i=\Root$}\;.
\end{cases}
\end{equation*}
%
%
%

If after coming back $m$ consecutive times to $i$ we  have added $h\le m$ nodes  (necessarily to $i$), then $B_{t_0+2m}\!(i)\le B_{t_0}\!(i)+h$, where the equality holds for $i \neq \Root$. But then, because $B_t(i)/d_t(i)$ is smaller than $1$, it holds.
\[\frac{B_{t_0+2m}(i)}{d_{t_0+2m}(i)} \le \frac{B_{t_0}\!(i)+h}{d_{t_0}\!(i)+h} \le \frac{B_{t_0}\!(i)+m-1}{d_{t_0}\!(i)+m-1}  \]
It then follows:
 \begin{align}\label{eq:return}
\mathbb{P}\big(\W_{t_0 +2}=i, & \W_{t_0 +4}=i,\cdots, \W_{t_0 +2k}=i \mid \W_{t_0}=i\big) = \frac{B_{t_0}\!(i)}{d_{t_0}\!(i)}\cdot \frac{B_{t_0+2}(i) }{d_{t_0+2}(i) }\cdot \ldots \cdot \frac{B_{t_0+2k}(i) }{d_{t_0+2k}(i) }\nonumber\\
	& \le \frac{B_{t_0}\!(i)}{d_{t_0}\!(i)}\cdot \frac{B_{t_0}\!(i) + 1}{d_{t_0}\!(i) + 1}\cdot \ldots \cdot \frac{B_{t_0}\!(i) + (k-1)}{d_{t_0}\!(i) + (k-1)}.
\end{align}
We observe that the equality holds, for $i\neq \Root$, when $\s=2$ and the walker indeed adds a new leaf every time it bounces back to vertex $i$.
Also,  we have that    $B_{t_0}\!(i)\leq d_{t_0}\!(i) - 1/2$, and this holds    regardless of whether $i=\Root$  or not. 
%
%
%
%
For $i\neq \Root$, the inequality $B_{t_0}\!(i)\leq d_{t_0}\!(i) - 1/2$  follows from the observation that if the random walk is in $i$ at time $t_0$, the vertex $i$ has at least one neighbour  which is not a leaf (it might be the only one) and therefore the latter will have degree  at least $2$. Hence, $\sum_{j \in \mathcal{N}_{t_0}\!(i)}1/d_{t_0}\!(j)\leq d_{t_0}\!(i) -1 + 1/2= d_{t_0}\!(i) -1/2$.
For $i=\Root$ instead, the inequality follows from 
the fact that $B^\Root_0\leq d_{t_0}\!(\Root)-2 + 4/d_{t_0}\!(\Root)$, together with $d_{t_0}\!(\Root)\geq 3$ because the first node at time $t=\s$ is necessarily added to the root. 
Applying the  bound $B_{t_0}\!(i)\leq d_{t_0}\!(i) - 1/2$ to each factor appearing in 
Equation~\eqref{eq:return}, we obtain
\begin{align*}
\mathbb{P}\big(\W_{t_0 +2}=i, & \W_{t_0 +4}=i,\cdots, \W_{t_0 +2k}=i  \mid \W_{t_0}=i\big)
\leq \prod_{j=0}^{k-1} \frac{2d_{t_0}\!(i) + (2j -1)}{2(d_{t_0}\!(i)+j)}= \prod_{j=d_{t_0}\!(i)}^{d_{t_0}\!(i)+ k-1} \frac{2j -1}{2j} \\
 &= \frac{\left( 2(d_{t_0}\!(i)+k-1)\right)! \; (d_{t_0}\!(i)-1)!^ 2 }{4^k \;(2d_{t_0}\!(i)-2)!\; (d_{t_0}\!(i)+k-1)!^2} 
= \frac{(d_{t_0}\!(i)-1)!^ 2}{(2d_{t_0}\!(i)-2)!}\cdot
\frac{\left( 2(d_{t_0}\!(i)+k-1)\right)!}{4^k \;  (d_{t_0}\!(i)+k-1)!^2} 
\;.								
\end{align*}
Using the bounds
  $\sqrt{2 \pi} n^{n+1/2} e^{-n} \le n! \le  e n^{n+1/2} e^{-n}$ (holding for all positive integer $n$), we have  
\[
\mathbb{P}\big(\W_{t_0 +2}=i, \W_{t_0 +4}=i,\cdots, \W_{t_0 +2k}=i \mid \W_{t_0}=i\big) \leq 
\begin{cases}
 \left(\frac{\mathrm{e}}{\sqrt{2\pi}}\right)^3 \frac{\sqrt{d_{t_0}\!(i)-1}}{\sqrt{d_{t_0}\!(i) + k -1}}
& \text{ if $d_{t_0}\!(i)\geq 2$, } 
 \\
 \left(\frac{\mathrm{e}}{\sqrt{2}\pi}\right) \frac{1}{\sqrt{k}} & \text{ if $d_{t_0}\!(i)=1$.} 
\end{cases}
\]
\qed			
\end{proof}

The result of Lemma~\ref{lem:star-escape} guarantees that the random walk will not get stuck  bouncing back in a particular vertex. In particular,  it implies   that   the first time at which the random walk stops bouncing back to the same vertex is finite almost surely. 

\begin{corollary}\label{cor:exit-time}
Let $\s$ be even, $t_0 \in 2\mathbb{Z}_{\geq 0}$ and  define $\tau \IsDef \inf \{n\geq 1: \W_{t_0+2n}\neq \W_{t_0}\}$, i.e.,  the first time the walker does not come back to the initial vertex after two steps. It holds 
\[
\mathbb{P}\left( \tau < \infty \right)=1\;.
\]
\end{corollary}
\begin{proof}
Note that the distribution of $\tau$ depends on the step parameter $\s$ as well as on the neighborhood of the walker at time $t_0$. However, Lemma~\ref{lem:star-escape} assures that for all $\s$ even, $\mathbb{P}\left( \tau > k\right)= \mathcal{O} \left(k^{-1/2}\right)$. Therefore, using that the sequence of events $\{\tau\leq k\}_k$ is increasing, we have 
\begin{align*}
\mathbb{P}\left( \tau <\infty \right)= \lim_{k\uparrow \infty}\mathbb{P}\left( \tau \leq k\right)= 
1- \lim_{k\uparrow \infty} \mathbb{P}\left( \tau > k\right)    =1 \;. 
\end{align*}
\end{proof}

%% file: root_visits_s_even.tex
\subsection{Recurrence of \Name{} for $\s$ even}\label{sec:change-of-parity}

Recall that in the \Name{} model the initial  node (root) has a self-loop. This local feature at the root plays a very prominent role in the model  when $s$ is even, as we now discuss. Let the level of a node in the generated tree denote its distance to the root. Note that the root is the only node at level 0, while all its neighbors are at level 1. Similarly we define the level of the random walk at time $t$ as the distance from $W_t$ to the root, denoted by $d(W_t,r)$. 

\begin{definition}
We say that the random walk at time $t$ is even $\iff$ $d(W_t,r) + t$ is even, and odd otherwise.
\end{definition}

The parity of the random walk has important consequences on the behaviour of the model when $s$ is even. In particular, as long as the random walk is even (resp., odd) new vertices can only be added to even (resp. odd) levels. However, if the random walk changes its parity once (or an odd number of  times) between two node additions, the next node will be added to a level with different parity. Clearly, changing parity  an even number of times between two node additions does not change the parity of the level to which nodes are added.

Note that the parity of the random walk can only change if the random walk traverses the self-loop. In fact, the latter is the only case in which the distance from the root stays constant and the time increases by one. For all other random walk steps instead, the parity does not change because the time always increases by one while the distance  either increases or decreases by one. 

We say that the random walk \emph{changes parity} whenever it traverses the self-loop.
The  change of parity is fundamental for the growing structure of the tree. Consider the addition of a node $i$ to the tree. A subsequent new node can only be connected to $i$ after the random walk changes its parity. Thus, once added to the tree, a new node can only receive a child node if the random walk changes parity after it has been added. This, in particular, implies that  the set of nodes that can receive a new node is finite and stays constant until the random walk changes its parity.
Therefore, in order to grow the tree to deeper levels the random walk must change its parity. Will the random walk change its parity a finite number of times? If so, the tree would have a finite depth. 
It is not hard to see that if the random walk visits the root infinitely many times than it must change its parity an infinite number of times. 
Thus, showing that the random walk is recurrent will also implies that the random walk changes its parity an infinite number of times almost surely. This is a necessary condition for the tree to grow its depth unbounded. 

\medskip

Before presenting the main theorem, we provide a preliminary result which relates the visits to a node to the visits to its neighbors. In particular, we show that if the random walk visits a node infinitely often, then it also visits any neighbors  infinitely often. 

\begin{lemma}\label{lem:neighbor}
Let $i$ be a vertex of the graph and $(i,j)$ an edge of the graph. If $i$ is recurrent, then the random walk traverses $(i,j)$ infinitely many times almost surely.  
\end{lemma}

\begin{proof}
Let $J_{t_1}^{t_2}\!(i)=\sum_{k=t_1}^{t_2} \mathbbm{1}(W_k=i)$  be the number of times the random walk visits visits node $i$ between $t_1$ and $t_2$. Similarly we denote by   
$J_{t_1}^{t_2}\!(i,j)=\sum_{k=t_1}^{t_2} \mathbbm{1}(W_k=i)\mathbbm{1}(W_{k+1}=j)$ the number of times the random walk traverses the edge $(i,j)$ in the direction from $i$ to $j$. We observe that $\mathbb P(J_t^{\infty}(i) =\infty | \{\mathcal G_t, W_t\}= \{G,w\} )=1$ for any possible configuration $\{G,w\}$ reachable by the stochastic process $\{\mathcal G_t, W_t\}$ at time $t$ (i.e. such that $\mathbb P(\{\mathcal G_t, W_t\}=\{G,w\})>0$. In fact if it were not the case, it would follow 
\[\mathbb P(J_0^{\infty}(i)<\infty) > \mathbb P(J_t^{\infty}(i)<\infty | \{\mathcal G_t, W_t\}= \{G,w\} ) \times \mathbb P(\{\mathcal G_t, W_t\}=\{G,w\}) >0, \]
contradicting the hypothesis that $i$ is recurrent.

Let $t_0$ be a time such that the nodes $i$, $j$ and the link $(i,j)$ belong to the graph $G_0 =\mathcal G_{t_0}$. 
We want to show that $J_{t_0}^\infty(i,j)=\infty$ with probability one.
%


Let $(\omega_t)_{t\geq t_0}$ be a sequence of independent uniform random variables over $[0,1]$, that we can use to determine which edge the random walk traverses at any time and then the evolution of the stochastic process $\{\mathcal G_t, W_t\}$ for $t\ge t_0$. 
In particular we can define a sequence of random variable $(\xi_t)_{t\geq t_0}$ that determines if the random walk traverses the edge $(i,j)$ if it is at node $i$ at time $t$:
\[\xi_t = \mathbbm{1}\!\left(\omega_t \in \left[0,\frac{1}{d_t(i)}\right]\right)\]
We have then 
\begin{align*}
J_{t_0}^t(i,j) = \sum_{k=t_0}^{t} \mathbbm{1}(W_k=i)  \, \mathbbm{1}(W_{k+1}=j)
= \sum_{k=t_0}^{t} \mathbbm{1}(W_{k}=i)\, \xi_k \;.
\end{align*}
Note that $(\xi_t)_t$ depend in general on the whole history  of the random walk till time $n$, because that history determines the degree of vertex $i$  at time $n$. 

For $h < J_{t_0}^\infty(i)+1$, let $t_h$ be the random time instants at which the random walk visits node $i$ for the $h$-th time after $t_0$. If $J_{t_0}^\infty(i) < \infty$, then let $t_h=t_{J_{t_0}^\infty(i)}+h$ for $h > J_{t_0}^\infty(i)$. We now define a new sequence of random variables as follows:
\[ \xi'_h = \mathbbm{1}\!\left(\omega_{t_h} \in \left[0,\frac{1}{d_{t_0}(i)+h}\right]\right)
\]
We observe that the variables  $(\xi'_h)_h$ are independent and the variable $\xi'_h$   is coupled with the variable $\xi_{t_h}$ through $\omega_{t_h}$. In particular $h< J_{t_0}^\infty(i)+1$ it always holds $\xi'_h \le \xi_{t_h}$ for  because the degree of node $i$ may have increased at most of $h$ after $h$ visits, i.e.~$d_{t_h}(i) \le d_{t_0}(i)+h$. Then for each possible path of the stochastic process, it holds:
\begin{align}
\label{e:J_ineq}
J_{t_0}^\infty(i,j) = \sum_{k=t_0}^{\infty} \mathbbm{1}(W_{k}=i)\, \xi_k \ge \sum_{k=1}^{J_{t_0}^\infty(i)} \xi'_k\;.
\end{align}

We now observe that
\begin{align*}
	\mathbb P(J_{t_0}^\infty(i,j)& <\infty)  \le \mathbb P\left(\sum_{h=1}^{J_{t_0}^\infty(i) } \xi'_h < \infty\right) \\
	& =  \mathbb P\left(\left\{\sum_{h=1}^{J_{t_0}^\infty(i) } \xi'_h < \infty \right\} \cap \left\{  J_{t_0}^\infty(i) =\infty \right\} \right ) +  \mathbb P\left(\left\{\sum_{h=1}^{J_{t_0}^\infty(i) } \xi'_h < \infty \right\} \cap \left\{  J_{t_0}^\infty(i) <\infty \right\} \right )\\
	& =  \mathbb P\left(\left\{\sum_{h=1}^{\infty } \xi'_h < \infty \right\} \cap \left\{  J_{t_0}^\infty(i) =\infty \right\} \right ) +  \mathbb P\left(\left\{\sum_{h=1}^{\infty } \xi'_h < \infty \right\} \cap \left\{  J_{t_0}^\infty(i) <\infty \right\} \right )\\
	& = \mathbb P\left(\sum_{h=1}^{\infty} \xi'_h < \infty \right) = 0.
\end{align*}
The inequality follows from Equation~\eqref{e:J_ineq}. In the second equality we have replaced $J_{t_0}^\infty(i)$ with infinity in the two sums: the first time this is permitted because the event of interest is a subset of $\{J_{t_0}^\infty(i) =\infty\}$, the second time because  $i$ is recurrent and then $\mathbb{P}\left(J_{t_0}^\infty(i) < \infty \right)=0$.  The last equality follows from applying Borel-Cantelli to the sequence of independent events $\{\xi'_h=1\}$. Indeed, $\sum_{h=1}^\infty \mathbb P(\xi'_h=1)= \sum_{h=1}^\infty \frac{1}{d_{t_0}+h} = \infty$ and then $\xi'_h=1$ infinitely often with probability one.
We can then conclude that $\mathbb P\left(J_{t_0}^\infty(i,j)=\infty\right)=1$.\qed
\end{proof}

\begin{corollary}
Let $i$ be a vertex of the graph and $j$ a neighbor of $i$. If $i$ is recurrent then $j$ is recurrent.
\end{corollary}

 \medskip
 
We can now state the main theorem of this section.

\begin{theorem}\label{thm:recurrence}
In the \Name{}  model with even step parameter, all vertices of the graph are recurrent. 
\end{theorem}

\begin{proof}
First we notice that to prove the theorem it is enough to show that the root is recurrent. 
In fact, given  an arbitrary node $i$ in the graph,  let  $(\Root=j_0, j_1, \dots, j_K=i)$ be the unique path from the root to $i$. If the root is recurrent then,  
using  Lemma~\ref{lem:neighbor},   we can conclude that  node $j_1$ 
 is recurrent. Iterating the reasoning along the nodes of the path, we have that  $i$ is recurrent.

We now  show that the root is recurrent.
Let us  define the following sequence of time instants  $\sigma_k \IsDef \inf\{t>\sigma_{k-1}: W_t=r\}$ and $\sigma_0\equiv 0$, i.e.~$\sigma_k$ is the time of the $k$-th visit to the root, if finite, or $\sigma_k=+\infty$ if the random walk visits the root less than $k$ times. The recurrence of the root is equivalent to   $\mathbb{P}(\sigma_k < \infty)=1$, for  all $k$. 
We proceed  by induction on $k$ and, 
 assuming  $\sigma_{k-1}< \infty$ almost surely, we show that   $\mathbb{P}(\sigma_k < \infty)=1$. 
 
By definition $W_{\sigma_{k-1}}=\Root$. If $W_{\sigma_{k-1}+1}=\Root$, then $\sigma_k$ is also finite. We then consider the case $W_{\sigma_{k-1}+1}\neq\Root$ and look at $\mathbb{P}(\sigma_k < \infty\mid W_{\sigma_{k-1}+1}\neq r)$. The random walk has then moved to one of the children of the root, and it  needs to pass by the root in order to traverse the loop and change parity. Until this does not happen, the parity of the walker is constant,  the level of newly added nodes will be opposite to that of the random walk, while the set of nodes with the same parity will not change. 

We assume for the moment that $\sigma_{k-1}$ is even and then the random walk is even in the interval $[\sigma_{k-1},\sigma_{k}]$. We look at the random walk every two steps and, for $t\in [\sigma_{k-1},\sigma_{k}]$,  define the process $Y_t=W_{2t}$, whose possible values are the nodes at even levels (including $\Root$) and then a finite set. The process $Y_t$ is a non-homogeneous Markov chain, because the addition of new nodes changes the transition probabilities, but we can define a homogeneous embedded Markov chain as follows.
Let $\phi_k\IsDef \inf \{t > \phi_{k-1}: Y_t \neq Y_{\phi_{k-1}}\}$.
The time instants $\phi_k$ are finite almost surely because of Corollary \ref{cor:exit-time}. We can then define the process $Z_k=Y_{\phi_k}$ and introduce the stopping time $\eta\IsDef \inf \{t> 0 : Z_t=r\}$, the first  time  the process $Z_k$ returns to the root. If $\mathbb{P}(\eta < \infty)=1$, then $\mathbb{P}(\sigma_k < \infty\mid W_{\sigma_{k-1}+1}\neq r)=1$.
The process $\{Z_k\}_k$   is an irreducible  time homogeneous Markov chain and its state space if finite (it is the same of $\{Y_t\}_t$). The time homogeneity follows from noticing that the transitions where $Y_t$ changes its state are determined by the graph configuration at time $\sigma_{k-1}$ and do not change afterwards because of the addition of new nodes (what changes it's the distribution of the sojourn times $\phi_k-\phi_{k-1}$).  
  An homogeneous irreducible Markov chain on  a finite state space  is recurrent, thus $\mathbb{P}(\eta < \infty \mid Z_0=r)=1$, and the claim follows.
  
 If $\sigma_{k-1}$ is odd, then $r$ does not belong itself to the set of possible values of $W_{2t}$ and then $Y_t$ and $Z_k$, but we can reason as follows. We imagine to cut the self-loop and add to a node $\Root'$ that is the new root of the graph connected by two edges to $\Root$. $\Root'$ is a possible value for $W_{2t}$ and reasoning as above we can conclude that with probability one the random walk will reach $\Root'$ in a finite time. But, in order to arrive to $\Root'$ the walker has to pass by $\Root$. It follows again that $\sigma_k< \infty$ with probability one.

\qed
\end{proof}

\begin{remark}
Theorem~\ref{thm:recurrence} and Lemma~\ref{lem:neighbor} imply that the random walk changes parity infinitely many times almost surely.
\end{remark}

%% file: density_s_even.tex
\subsection{Fraction of leaves in the graph with $\s$ even}

In this section we provide an asymptotic lower bound for the fraction of leaves in \Name{} model for every $\s$ even.
For $\s=2$, we prove that the fraction of leaves goes to one as the size of the graph goes to infinity. 
This implies that when $s=2$ the graph generated by \Name{} does not follow a power law degree distribution, 
contrasting with the preferential attachment model of Barab{\'a}si-Albert~\cite{barabasi1999emergence, bollobas2001degree}. 

The related model of Saram{\"a}ki and Kaski~\cite{saramaki2004scale} exhibits a similar characteristic, and was shown that under a given parameter choice the fraction of leaves also converges to one, asymptotically~\cite{cannings2013random}. This theoretical result was important since prior experimental evidence (wrongly) suggested that the model of Saram{\"a}ki and Kaski always led to a power law degree distribution, as in the model of Barab{\'a}si-Albert.

%
 %
 

Before stating the main theorem we introduce an auxiliary result which will  be instrumental in its proof. Let $L^\s_n$ denote the number of leaves in the graph at time $\s n$ (i.e., soon after the $n$-th vertex has been added).
Note that  $L^\s_n$ cannot decrease with the addition of new vertices. A new node always joins the network as a leaf, and either connects to an existing leaf or a non-leaf. In the former case, the number of leaves does not increase, whereas in the latter it increases by one. 

If the addition of the $n$-th vertex does not increase the number of leaves, we know that at time $\s n$ the random walk resides on a vertex which has only two neighbours; a parent and a leaf (the node just added).
 Let ${\rm T}$ be the random variable denoting the first time the random walk visits the parent of a vertex $i$ after adding the first leaf to $i$.  
It turns out that in order to provide an asymptotic lower bound for the fraction of leaves it is enough to compute the expected value of ${\rm T}$. 
The random variable ${\rm T}$ takes value in the positive and odd integers and its distribution only depends on $\s$. Specifically, we have 
%
\begin{equation*}
\mathbb{P}({\rm T}=2k + 1) = 
\begin{cases}
\left(\frac{1}{2}\right)^{k+1} & \quad\text{ for $k \in \{0,\ldots, \frac{\s}{2}-1\}$} 
\\
\left(\frac{1}{2}\right)^{\frac{\s}{2}} \left(\frac{2}{3}\right)^{k-\frac{\s}{2}} \frac{1}{3} & \quad \text{ for $k \in \{\frac{\s}{2},\ldots, 2 \frac{\s}{2}-1\}$} 
\\
\left(\frac{1}{3}\right)^{\frac{\s}{2}} \left(\frac{3}{4}\right)^{k-2\frac{\s}{2}} \frac{1}{4} & \quad \text{ for $k \in \{2\frac{\s}{2},\ldots, 3 \frac{\s}{2}-1\}$} 
\\
\left(\frac{1}{4}\right)^{\frac{\s}{2}} \left(\frac{4}{5}\right)^{k-3\frac{\s}{2}} \frac{1}{5} & \quad \text{ for $k \in \{3\frac{\s}{2},\ldots, 4 \frac{\s}{2}-1\}$} 
\\
\left(\frac{1}{5}\right)^{\frac{\s}{2}} \left(\frac{5}{6}\right)^{k-4\frac{\s}{2}} \frac{1}{6} & \quad \text{ for $k \in \{4\frac{\s}{2},\ldots, 5 \frac{\s}{2}-1\}$} 
\\
\qquad \qquad \vdots & \qquad \ldots 
\end{cases}
\end{equation*}

In a more compact form, using that for any $k \in \{0,1,\dots,\}$ there exist  unique $q_k \in \{0,1,\dots\}$ and  $r_k \in \{0,1, \dots, \s/2-1\}$ such that
$k= q_k \frac{s}{2} + r_k$, we can write 
\begin{align}\label{eq:distr}
\mathbb{P}({\rm T}=2k + 1) = \left( \frac{1}{q_k + 1} \right)^{\s/2}  \left( \frac{q_k + 1}{q_k + 2} \right)^{r_k}\frac{1}{q_k + 2}
\;,
\end{align}
and, in particular, $\mathbb{P}({\rm T}\geq 2k + 1) = \left( \frac{1}{q_k + 1} \right)^{\s/2}  \left( \frac{q_k + 1}{q_k + 2} \right)^{r_k}$.
If follows from Corollary~\ref{cor:exit-time} that ${\rm T}$ is finite almost surely for every $\s$ even. It can also be computed directly from Equation~\eqref{eq:distr} that $\mathbb{P}({\rm T}< \infty)=\sum_{k=1}^\infty \mathbb{P}({\rm T}=2k + 1)=1$.
In the next lemma we compute the expected value of ${\rm T}$. 
\begin{lemma}\label{lem:expect}
For $\s$ even  it holds that 
\[
\mathbb{E} ({\rm T})= 
1 + 2\,\zeta\left(\s/2\right)\;,
\]
where $\zeta(z)=\sum_{m=1}^\infty\frac{1}{m^{z}}$ is the Reimann zeta function. 
\end{lemma}

Note that  $\mathbb{E} ({\rm T}) = + \infty$ for  $\s=2$, whereas  $\mathbb{E} ({\rm T}) < \infty$, for $s\geq 4$.

\begin{proof}
Recall  that if $X$ is a random variable which takes only positive integer values then  $\mathbb{E}(X)= \sum_{i=1}^\infty \mathbb{P}(X\geq i)$.
The random variable ${\rm T}$ only takes odd integer values, which implies that $\mathbb{P}({\rm T}\geq i)=\mathbb{P}({\rm T}\geq i+1)$, for every even $i$. Therefore, when summing over all integer values strictly bigger than one, we are summing twice the contributions of the odd values and  \begin{equation}\label{eq:sum_prob}
\mathbb{E} ({\rm T})=\sum_{i=1}^\infty \mathbb{P}({\rm T}\geq i)= \mathbb{P}({\rm T}\geq 1) + 2 \sum_{k=1}^{\infty}  \mathbb{P}({\rm T}\geq 2k+1)=2\sum_{k=0}^{\infty} \mathbb{P}({\rm T}\geq 2k+1)-1\;, 
\end{equation}
where, in the last equality we used that $\mathbb{P}({\rm T}\geq 1)=1$.
The probabilities appearing on the right-hand side of Equation~\eqref{eq:sum_prob} satisfy,
$\mathbb{P}({\rm T}\geq 2k + 1) = \left( \frac{1}{q_k + 1} \right)^{\s/2}  \left( \frac{q_k + 1}{q_k + 2} \right)^{r_k}$,
where, $q_k \in \{0,1,\dots\}$ and  $r_k \in \{0,1, \dots, \s/2-1\}$ are such that $k= q_k \frac{s}{2} + r_k$. 
%
Therefore, 
\begin{equation*}
\begin{split}
&\sum_{k=0}^\infty \mathbb{P}({\rm T}\geq 2k + 1) = \sum_{q=0}^{\infty}  \left(\frac{1}{q+1}\right)^{\s/2} \;\sum_{r=0}^{\s/2-1} \left(\frac{q+1}{q+2}\right)^{r} 
=\\
& = \sum_{m=1}^\infty (m+1)\left(\frac{1}{m}\right)^{\s/2} \left( 1- \left(\frac{m}{m+1}\right)^{\s/2}\right) 
=  \sum_{m=1}^\infty (m+1)\left(\frac{1}{m^{\s/2}} - \frac{1}{(m+1)^{\s/2}}\right)
\\
&
= 1+ \sum_{m=1}^\infty m\left(\frac{1}{m^{\s/2}} - \frac{1}{(m+1)^{\s/2}}\right)
=  1 +  \sum_{m=1}^\infty \frac{1}{m^{\s/2}}
\end{split} 
\end{equation*}
Overall, we obtain 
\[
\mathbb{E}({\rm T})= 2\left( 1 +  \sum_{m=1}^\infty \frac{1}{m^{\s/2}} \right) -1 = 1 + 2 \sum_{m=1}^\infty \frac{1}{m^{\s/2}}
\]
\qed
\end{proof}

\medskip

We can now state the main theorem of this section.

%

\begin{theorem}\label{th:density_leaves}
Let $\s$ be even and $L_n^\s$ denote the number of leaves in the graph of size $n$.  It holds that   
\[
\liminf_{n\uparrow \infty}\frac{L^\s_n}{n}\geq 1-\frac{1}{\mathbb{E}({\rm T} )} \qquad  \text{almost surely}\;,
\]
where $\mathbb{E}({\rm T})$ is given in Lemma~\ref{lem:expect}.
\end{theorem}

The above theorem immediately implies the following results:
\begin{corollary}\label{th:density_leaves_s}
Let $L_n^\s$ denote the number of leaves in the graph of size $n$, then
\begin{itemize}
\item
for $s=2$, $\lim_{n\uparrow \infty}\frac{L_n^\s}{n} = 1$, a.s.
\item
for every $s$ even, $\liminf_{n\uparrow \infty}\frac{L_n^\s}{n}\geq \frac{2}{3}$, a.s.
\end{itemize}
\end{corollary}
The above follows since $\mathbb{E}({\rm T}) = + \infty$ for $s=2$, and
$\mathbb{E}({\rm T})\geq 3$ for every $s$ even. 
%

\begin{proof}
%
%
 
%
%

%

For $l=1,2,\ldots$, let $A_l\IsDef \s \cdot \inf \{n>A_{l-1} : L^\s_n -L^\s_{n-1}=0 \}$ denote the time of the $l$-th vertex addition that does not increase the number of leaves in the graph (where $A_0 \equiv 0$). 
If there exists an $l$ such that $A_l=+\infty$ we set $A_{l'}=\infty$ for all $l'>l$.
For all sample paths such that there exists an $l$ with $A_l= +\infty$, we have $\lim_{n \uparrow \infty}\frac{L_n^\s}{n} = 1$, which clearly implies  $\liminf_{n\uparrow \infty}\frac{L^\s_n}{n}\geq 1-\frac{1}{\mathbb{E}({\rm T} )}$. In this case, in fact,  
there exists a finite time such that  every vertex added after it will increase the total number of leaves.

We consider now all sample paths such that $A_l< \infty$,  for all $l$. 
To prove the claim  it is enough to show that, conditioned on the latter set of sample paths, 
$\liminf_{n\uparrow \infty}\frac{L_n^\s}{n}\geq  1-\frac{1}{\mathbb{E}({\rm T})}$ almost surely. 
%
%
Let $T_l = A_{l} - A_{l-1}$ denote the time between the $(l+1)$-th and $l$-th such node additions. 
Define  $S_m \IsDef T_1 + T_2 + \cdots + T_m$ and $N^\s_n= \max \{m : S_m \leq n\}$. 
Note that $N^\s_n$ is the  number of non-leaf vertices in the graph at time  $\s n$ (not counting the root). Thus, we have that $N^\s_n +1 = n - L^\s_n$, and to prove the claim it suffices to provide an upper bound for $\limsup_{n\uparrow \infty}\frac{N^\s_n}{n}$. 

We use renewal theory for the counting process $N^\s_n$. However, since $T_l$ are not independent nor identically distributed, we cannot directly apply renewal theory.
We circumvent this limitation with the following argument.  By definition, $A_l$ are the times at which the number of leaves does not increase. This means that at these times the random walk resides in a vertex of degree two and its neighbors are its parent and a leaf node (the new added node). 
Let $i_l$ denote the vertex where the random walk resides at time $A_l$, i.e., $W_{A_l}=i_l$, and denote by $p_{i_l}$ the parent of vertex $i_l$ (note that $p_{i_l}$ is well defined because $i_l$ can never be the root). Let us define $\overline{T}_l=\inf\{t> A_{l} : \W_{t}=p_{i_l} \} - A_l$, i.e., the amount of time until the random walk visits the parent of vertex $i_l$ after adding  a leaf to  vertex $i_l$ for the first time (which occurs at time $A_l$).
Due to the fact that at time $A_l$ the random walk is in a vertex of degree two, the distribution of $\overline{T}_l$ does not depend on the specific $l$ and it is independent from the past. In particular $\overline{T}_l$ has the same distribution of ${\rm T}$ given in Equation~\eqref{eq:distr} and its expected value is given in Lemma~\ref{lem:expect}.

Consider $\{\overline{T}_l\}_{l\in \mathbb{N}}$  a sequence of i.i.d. random variables distributed as ${\rm T}$. 
%
%
For every $l=1,2,\ldots$, it holds that $\overline{T}_l < T_{l}$. This follows since to add a new vertex to an existing leaf (which occurs at time $A_{l+1}$), the random walk must visit the parent node of $i_l$. 
Let us define $\overline{S}_m\IsDef \overline{T}_1 + \overline{ T}_2 + \cdots + \overline{ T}_m$  and  let $\overline{N}^\s_n=\max \{m : \overline{S}_m \leq n\}$ denote the corresponding counting process. 
Since $\overline{T}_l<T_l$, for every $m$ we have that $\overline{S}_m<S_m$ which implies $\overline{N}^s_n>N^s_n$ and consequently $0 \leq \frac{N^\s_n}{n}\leq \frac{\overline{N}^\s_n}{n}$, for every $n$.
Given that $\lim_{n\uparrow \infty} \overline{S}_n=\infty$ a.s. and $\lim_{n\uparrow \infty} \overline{N}^\s_n=\infty$ a.s.,  we can apply the  Strong Law of Large Numbers  for $\overline{N}^\s_n$, stating that $\lim_{n\uparrow\infty}\frac{\overline{N}^\s_n}{n}=1/\mathbb{E}({\rm T})$ a.s., where $\mathbb{E}({\rm T})$ is given in Lemma~\ref{lem:expect}.
Therefore, we obtain that $\limsup_{n 
\uparrow \infty} \frac{N^\s_n}{n}\leq  \frac{1}{\mathbb{E}({\rm T})}$ a.s., which implies $\liminf_{n\uparrow \infty}\frac{L_n^\s}{n}\geq  1-\frac{1}{\mathbb{E}({\rm T})}$.
 \qed
\end{proof}

%% file: conclusion.tex

\section{Final Remarks}\label{sec:conclusion}

The \Name{} model illustrates the powerful interplay of mutually coupled dynamics, in this case random walking and graph building (tree). Interestingly, the result of this interplay is fundamentally governed by the parity of its single parameter $\s$. In particular, the dichotomy between transience and recurrence for the random walk and between power law and exponential degree distributions for the tree. Indeed, \Name{} can be analyzed from two different perspectives: the walker behavior and the graph behavior. Given its mutual dependency, the two perspectives are also fundamentally intertwined as we have shown, without one necessarily driving the other. 

While Theorem~\ref{thm:s_1} holds only for $s=1$, we conjecture that such result is true for all $\s$ odd. Extensive simulations support this conjecture, showing that the degree distribution is bounded by an exponential and that distances from the root grows linearly~\cite{amorim2016growing}. Unfortunately, the coupling technique applied in the proof for the case $s=1$ cannot be extended to larger values of $s$. 

A variation of \Name{} named BGRW (Bernoulli Growth Random Walk) model has recently been proposed and analyzed~\cite{figueiredo2017building}. In BGRW, after every walker step a new node with degree one is connected to the current walker position with probability $p$. The main result states that the random walk in BGRW is transient for any $p>0$. This suggests that recurrence in \Name{} is an artifact of the inherent structural limitations imposed by the even parity of $\s$. Moreover, it also supports the above conjecture (walker in \Name{} is transient for all $\s$ odd). 

Last, for $\s=2$ the degree distribution of \Name{} does not follow a power law degree distribution, as the fraction of leaves (degree 1 nodes) goes to 1 as the network grows. However, conditioned on nodes with degree larger than 1, the conditional degree distribution is lower bounded by a power law distribution. While the original Random Walk Model has global restarts (and thus is fundamentally different from \Name{}) this observation may reconcile the apparently conflicting results of \cite{saramaki2004scale} (which claims that degree distribution follows a power law, for all $\s$) and \cite{cannings2013random} (which proves that the fraction of leaves converges to 1 for $\s=1$).